\documentclass[11pt,reqno]{amsart}
\usepackage{amssymb,latexsym}
\usepackage{graphicx}
\usepackage{url}		

\calclayout

\makeatletter
\newcommand{\monthyear}[1]{%
  \def\@monthyear{\uppercase{#1}}}
\newcommand{\volnumber}[1]{%
  \def\@volnumber{\uppercase{#1}}}
\AtBeginDocument{%
\def\ps@plain{\ps@empty
  \def\@oddfoot{\@monthyear \hfil \thepage}%
  \def\@evenfoot{\thepage \hfil \@volnumber}}
\def\ps@firstpage{\ps@plain}
\def\ps@headings{\ps@empty
  \def\@evenhead{%
    \setTrue{runhead}%
    \def\thanks{\protect\thanks@warning}%
    \uppercase{The Fibonacci Quarterly}\hfil}%
  \def\@oddhead{%
    \setTrue{runhead}%
    \def\thanks{\protect\thanks@warning}%
    \hfill\uppercase{On odd perfect numbers}}%
  \let\@mkboth\markboth
  \def\@evenfoot{%
    \thepage \hfil \@volnumber}%
  \def\@oddfoot{%
    \@monthyear \hfil \thepage}%
  }%
\footskip=25pt
\pagestyle{headings}%
}
\makeatother

\newcommand{\Z}{{\mathbb Z}}

\theoremstyle{plain}
\numberwithin{equation}{section}
\newtheorem{thm}{Theorem}[section]
\newtheorem{theorem}[thm]{Theorem}
\newtheorem{lemma}[thm]{Lemma}

\begin{document}
\monthyear{Month Year}
\volnumber{Volume, Number}
\setcounter{page}{1}

\title{A note on odd perfect numbers}
\author{Jose Arnaldo B. Dris}
\address{Far Eastern University\\
                   Manila, Philippines}
\email{josearnaldobdris@gmail.com}

\author{Florian Luca}
\address{School of Mathematics\\
                  University of the Witwatersrand \\
                  Private Bag X3, Wits 2050\\
                  South Africa\\
                  \and
                  Centro de Ciencias Matem\'aticas\\
                  UNAM, Morelia, Mexico}
\email{florian.luca@wits.ac.za}

\begin{abstract}
In this note, we show that if $N$ is an odd perfect number and $q^{\alpha}$ is some prime power exactly dividing it, then $\sigma(N/q^{\alpha})/q^{\alpha}>5$. 
In general, we also show that if $\sigma(N/q^{\alpha})/q^{\alpha}<K$, where $K$ is any constant, then $N$ is bounded by some function depending on $K$.
\end{abstract}

\maketitle

\section{Introduction}

For a positive integer $N$ we write $\sigma(N)$ for the sum of the divisors of $N$. A number $N$ is {\it perfect} if $\sigma(N)=2N$.  Even perfect numbers have been characterized by Euler. Namely, $N$ is an even perfect number if and only if $N=2^{p-1}(2^p-1)$, where $2^p-1$ is prime. Hence, the only obstruction in proving that there are infinitely many of them lies with proving that there exist infinitely many primes of the form $2^p-1$.

We know less about odd perfect numbers. No example has been found, nor do we have a proof that they don't exist. If they exist, then they must have at least $7$ distinct prime factors, a result of Pomerance from \cite{Pom}. The bound $7$ has been raised to $9$ recently in \cite{Nie}. Brent et. al. \cite{Betal} showed that 
$N>10^{300}$. The exponent $300$ has been raised to $1500$ in the recent work \cite{Ochetal}.

Let $N$ be perfect and let $q^{\alpha}\| N$, where $q$ is prime. Recall that the notation $q^{\alpha}\| N$ stands for the power of $q$ exactly dividing $N$, namely $q^{\alpha}\mid N$ but $q^{\alpha+1}\nmid N$. Then
$$
2N=2q^{\alpha} \left(\frac{N}{q^{\alpha}}\right)=\sigma(N)=\sigma(q^{\alpha}) \sigma\left(\frac{N}{q^{\alpha}}\right),
$$
and since $q^{\alpha}$ is coprime to $\sigma(q^{\alpha})$, it follows that $\sigma(N/q^{\alpha})/q^{\alpha}$ is an integer divisor of $2N$. When $N=2^{p-1}(2^p-1)$ is even, then 
$$
\frac{\sigma(N/q^{\alpha})}{q^{\alpha}}=\left\{ \begin{matrix} 2 & {\text{\rm if}} & q=2\\
1 & {\text{\rm if}} & q=2^{p}-1.\\ \end{matrix}\right.
$$
Here, we study this statistic when $N$ is an odd perfect number. We prove:

\begin{theorem}
\label{thm:1}
If $N$ is an odd perfect number and $q^{\alpha}\| N$ is a prime power exactly dividing $N$, then $\sigma(N/q^{\alpha})/q^{\alpha}>5$.
\end{theorem}

This improves on a previous lower bound obtained by the first author in his M.~S.~thesis \cite{Dr}. 

The lower bound $5$ can likely be easily improved although it is not clear to us what the current numerical limit of this improvement should be. We leave this as a problem for other researchers.  In light of the above result, one may ask whether it could be the case that by imposing an upper bound on the amount $\sigma(N/q^{\alpha})/q^{\alpha}$, the number $N$ ends up being bounded as well. This is indeed so as shown by the following result.

\begin{theorem}
\label{thm:2}
For every fixed $K>5$, there are only finitely many odd perfect numbers $N$ such that for some prime power $q^{\alpha}\|N$ we have that $\sigma(N/q^{\alpha})/q^{\alpha}<K$. All such $N$ are bounded by some effectively computable number depending on $K$.
\end{theorem}

The proof of Theorem \ref{thm:1} is elementary. The proof of Theorem \ref{thm:2} uses the arguments from the proof of the  particular case Theorem \ref{thm:1}
together with two more ingredients. The first ingredient is a result of Heath--Brown \cite{HB} to the effect that an odd perfect number $N$ with $s$ distinct prime factors 
cannot exceed $4^{4^{s+1}}$. The second ingredient is a well-known result from the theory of Exponential Diophantine Equations (for the main results in this area, see \cite{ST}) regarding the largest prime factor of $f(n)$ for large $n$, where $f(X)\in \Z[X]$ is a polynomial with at least two distinct roots. 
 
One of the main tools for the proof of both Theorems \ref{thm:1} and \ref{thm:2} is the following result due to Bang (see \cite{Ba}).

\begin{lemma}
\label{lem:Bang}
Let $a>1$ be an integer. For all $n\ge 7$, there is a prime factor $p$ of $a^n-1$ which does not divide $a^m-1$ for any $1\le m<n$. Moreover, such a prime $p$ is congruent to $1$ modulo $n$.
\end{lemma}

Throughout the paper, we use $p,~q,~r$, $P$ and $Q$ with or without subscripts for prime numbers. 

\section{Preliminaries}

Put $N=p_1^{\lambda_1}\cdots p_s^{\lambda_s}q^{\alpha}$, where the primes $p_1,~\ldots,~p_s,~q$ are distinct and not necessarily ordered increasingly. We write 
\begin{equation}
\label{eq:1}
\sigma(p_i^{\lambda_i})=m_i q^{\beta_i},\quad i=1,\ldots,k,\quad {\text{\rm and}}\qquad \sigma(p_i^{\lambda_i})=q^{\beta_i},\qquad i=k+1,\ldots,s,
\end{equation}
where $m_i\ge 2$ for $i=1,\ldots,k$, and $\alpha=\beta_1+\cdots+\beta_s$. For both proofs of Theorem \ref{thm:1} and \ref{thm:2} we will need 
facts about equations \eqref{eq:1} for some $i$ either in $\{1,\ldots,k\}$ and with a fixed value of $m_i\ge 2$, or with $i\in \{k+1,\ldots,s\}$. Observe that $\lambda_i$ is 
even for all $i=k+1,\ldots,s$, and $\lambda_i$ is even for at most one $i\in \{1,\ldots,k\}$.

We treat first the case of a fixed $m_i\ge 2$. For simplicity, put $p:=p_i$, $\beta:=\beta_i$, $m:=m_i$, and $\lambda:=\lambda_i$ for some $i=1,\ldots,k$. Then the first equation \eqref{eq:1} for the index $i$ is
\begin{equation}
\label{eq:*}
\frac{p^{\lambda+1}-1}{p-1}=mq^{\beta}.
\end{equation}
Here, $p$ and $q$ are odd. We prove the following lemma.

\begin{lemma}
\label{lem:1}
In equation \eqref{eq:*}, we have $\lambda+1\le m^2$.
\end{lemma}

\begin{proof}
For a positive integer $n$ coprime to $p$ let $\ell_p(n)$ be the multiplicative order of $p$ modulo $n$. Let $u_n:=(p^n-1)/(p-1)$. Then $m$ divides $u_{\lambda+1}$. 
It is then well-known that $\lambda+1$ is divisible by the number $z_p(m)$ defined as
$$
z_p(m):={\text{\rm lcm}}[z_p(r^{\delta}),~r^{\delta}\| m],
$$
where 
$$
z_p(r^{\delta})=\left\{\begin{matrix} r^{\delta} & {\text{\rm if}} ~ p\equiv 1\pmod r,\\
\ell_p(r^{\delta}) & {\text{\rm otherwise}}.\\
\end{matrix} \right.
$$ 
Clearly, $2\le z_p(m)\le m$. Equality is achieved if and only if each prime factor of $m$ is also a prime factor of $p-1$. Assume that $\lambda+1>m^2$. Write $\lambda+1=z_p(m) d$, where $d>z_p(m)$.  We look at $u_d=(p^d-1)/(p-1)$ which is  a divisor of $mq^{\beta}$. If $d\ge 7$, then  by Lemma \ref{lem:Bang} there is a prime factor $P$ of $u_d$ which does not divide $u_{z_p(m)}$. Since all prime factors of $m$ divide $u_{z_p(m)}$, we must have that $P=q$. But then all prime factors of $u_{\lambda+1}$ divide either $q$; hence, $u_d$, or $m$; hence $u_{z_p(m)}$, contradicting Lemma \ref{lem:Bang}. So, we have a contradiction if $d\ge 7$. 

Thus, all prime factors of $u_d$ are among the prime factors of $u_{z_p(m)}$ and so $d\le 6$. In particular,  the prime factors of $d$ must be among the prime factors of  $z_p(m)$, for otherwise, namely if there 
is some prime factor $Q$ of $d$ which does not divide $z_p(m)$, then $u_Q=(p^Q-1)/(p-1)$ is a divisor of $u_d=(p^d-1)/(p-1)$ which is coprime to $u_{z_p(m)}$, a multiple of $m$, which in turn is false. Thus, all prime factors of $d$ are indeed among the prime factors of $z_p(m)$, and since $d>z_p(m)$, there is a prime 
factor $Q$ of $d$ which appears in the factorization of $d$ with an exponent larger than the exponent with which it appears in the factorization 
of $z_p(m)$. Hence, $d\ge Q^2$, and since $d\le 6$, we get that $Q=2$. This implies that $u_{\lambda+1}$ is a multiple of $u_4$; hence, a multiple of $4$, which is false. This shows that indeed $\lambda+1\le m^2$. 
\end{proof}

We next treat the case of $i\in \{k+1,\ldots,s\}$.

\begin{lemma}
\label{lem:2}
The equations
\begin{equation}
\label{eq:lemm2}
\frac{p^{\lambda+1}-1}{p-1}=q^{\beta}\qquad {\text{and}}\qquad p^{\lambda}\mid \frac{q^{\alpha+1}-1}{q-1}
\end{equation}
imply that $\alpha+1$ is a multiple of $p^{\lambda-1}$. 
\end{lemma}

\begin{proof}
The left equation in \eqref{eq:lemm2} is
$$
p^{\lambda}+\cdots+p^2+p=q^{\beta}-1,
$$
showing that $p\| q^{\beta}-1$. This implies easily that $p\| q^{\ell_{q}(p)}-1$. Now the conclusion follows immediately from the divisibility relation from the right hand side of equation \eqref{eq:lemm2}.
\end{proof}

Let $m:=m_1\cdots m_k$. We put $M:=\prod_{i=k+1}^s p_i^{\lambda_i}$. We label the numbers $\lambda_1,\ldots,\lambda_k$ such that $\beta_1\le \beta_2\le \cdots \le \beta_k$. Applying Lemma \ref{lem:2} for $i=k+1,\ldots,s$, we get that
\begin{equation}
\label{eq:alpha}
\alpha+1\ge \prod_{i=2}^s p_i^{\alpha_i-1}\ge \prod_{i=2}^s p_i^{\alpha_i/2}=M^{1/2}.
\end{equation}
Let $\Lambda:={\text{\rm lcm}}[\lambda_i+1:i=1,\ldots,k]$. Then $\Lambda\le \prod_{i=1}^k m_i^2=m^2$. Observe that $p_i^{\lambda_i}\equiv 1\pmod {q^{\beta_i}}$ for $i=1,\ldots,k$. In particular,
\begin{equation}
\label{eq:Lambda}
p_i^{\Lambda}\equiv 1\pmod {q^{\beta_i}}\qquad {\text{\rm for~all}}\quad i=1,\ldots,k.
\end{equation}

\begin{lemma}
\label{lem:3}
One of the following holds:
\begin{itemize}
\item[(i)] $q\mid m$;
\item[(ii)] $q^{\beta_i}<(2Mqm)^{(m^2+1)^{i}}$ for $i=1,\ldots,k$.
\end{itemize}

\begin{proof}
If $q\mid m$, we are through. So, suppose that $q\nmid m$. We write
$$
\frac{q^{\alpha+1}-1}{q-1}=\sigma(q^{\alpha})=\left(\frac{2M}{m}\right) p_1^{\lambda_1}\cdots p_k^{\lambda_k}.
$$
We raise the above equation to the power $\Lambda$ and use congruences \eqref{eq:Lambda} getting 
$$
\left(\frac{-1}{q-1}\right)^{\Lambda}\equiv \left(\frac{2M}{m}\right)^{\Lambda} \pmod {q^{\beta_1}}.
$$
Hence, $q^{\beta_1}\mid (2M(q-1))^{\Lambda}\pm m^{\Lambda}$. The last expression is nonzero, since if it were zero, we would get $2M(q-1)=m$, which is impossible because $2M(q-1)$ is a multiple of $4$, whereas $m$ is a divisor of $2N$, so it is not a multiple of $4$.  Thus, 
$$
q^{\beta_1}\le (2M(q-1))^{\Lambda}+m^{\Lambda}<(2M(q-1)+m)^{\Lambda}<
(2Mqm)^{\Lambda}<(2Mqm)^{m^2}.
$$ 
This takes care of the case $i=1$ of (ii). For the case of a general $i$ in (ii), suppose, by induction, that $i\ge 2$ and that we have proved that  
$$
q_j^{\beta_j}<(2Mqm)^{(m^2+1)^{j} }\quad {\text{\rm holds~for~all}}\quad  j=1,\ldots,i-1.
$$
Then $p_j^{\lambda_j}<\sigma(p_j^{\lambda_j})=m_j q^{\beta_j}$ for $j=1,\ldots,i-1$. Thus,
\begin{equation}
\label{eq:inter}
\frac{q^{\alpha+1}-1}{q-1}=\left(\frac{2Mp_1^{\lambda_1}\cdots p_{i-1}^{\lambda_{i-1}}}{m}\right) p_i^{\lambda_i}\cdots p_k^{\lambda_k}.
\end{equation}
We raise  congruence \eqref{eq:inter} to the power $\Lambda$ and reduce it modulo $q^{\beta_i}$ obtaining 
$$
q_i^{\beta_i}\mid (2M(q-1)p_1^{\lambda_1}\cdots p_{i-1}^{\lambda_{i-1}})^{\Lambda}\pm m^{\Lambda}.
$$
The last expression above is not zero since $2M(q-1)p_1^{\lambda_1}\cdots p_{i-1}^{\lambda_{i-1}}$ is a multiple of $4$ and $m$ is not. Hence,
\begin{eqnarray*}
q^{\beta_i} & \le & (2M(q-1)p_1^{\lambda_1}\cdots {p_{i-1}}^{\lambda_{i-1}})^{\Lambda}+m^{\Lambda}\\
& \le & m^{\Lambda}((2M(q-1)^{\Lambda}+1)(q^{\beta_1}\cdots {q_{i-1}}^{\beta_{i-1}})^{\Lambda} \\
& \le &  (2qMm)^{\Lambda} q^{\Lambda (\beta_1+\cdots+\beta_{i-1} ) }\\
& < & (2qMm)^{m^2(1+(m^2+1)+\cdots+(m^2+1)^{i-1})}\\
& < & (2qMm)^{(m^2+1)^i},
\end{eqnarray*}
which is what we wanted to prove.
\end{proof}
\end{lemma}

\section{The proof of Theorem \ref{thm:1}}

Since $m\le 5$ and at most one of the $m_i$s is even for $i=1,\ldots,k$, we get that $k\le 1$. Then Lemma \ref{lem:3} shows that either $q\le 5$, or
\begin{equation}
\label{eq:10}
q^{\beta_i}\le (20Mq)^{25^i}\qquad {\text{\rm for~all}}~i=1,\ldots,k.
\end{equation}
Assume that $q>5$. Then
\begin{eqnarray*}
q^{{\sqrt{M}}-1} & \le & q^{\alpha}=\left(\frac{2M}{m}\right) p_1^{\lambda_1}\cdots p_k^{\lambda_k}<2Mq^{\beta_1+\cdots +\beta_k}\\
& < & 2M (20qM)^{25}<
(20qM)^{26},
\end{eqnarray*}
leading to
$$
3^{{\sqrt{M}}-27}\le q^{{\sqrt{M}}-27}<(20M)^{26},
$$
which implies that $M<2\times 10^5$. Since $s\ge 8$ and $k\le 1$, there exists $j\in \{k+1,\ldots,s\}$ 
such that $p_j^{\beta_j}<M^{1/7}\le 6$. This is false because $p_j^{\beta_j}\ge 3^2=9$. Thus, $q\in \{3,5\}$. 

The equation from the right hand side of \eqref{eq:1} 
with $p:=p_2$, $\lambda:=\lambda_2$ and $\beta:=\beta_2$ becomes 
$$
\frac{p^{\lambda+1}-1}{p-1}=q^{\beta}.
$$
Observe that $\lambda+1$ is odd. If $\lambda+1\ge 7$, we get a contradiction from Lemma \ref{lem:Bang} because $q\le 5$. Thus, 
$\lambda\in \{2,4\}$ and we get one of the four equations
\begin{eqnarray*}
p^2+p+1 & = & 3^{\beta},\qquad p^4+p^3+p^2+p+1=3^{\beta},\\
p^2+p+1 & = & 5^{\beta},\qquad p^4+p^3+p^2+p+1=25^{\beta}.
\end{eqnarray*}
Arguments modulo $9$ and $25$ show first that the exponents $\beta$ from the above equations are in $\{0,1\}$, which immediately implies that none of the above equations has in fact any solutions.

\section{The Proof of Theorem \ref{thm:2}}

Here, we have $m=m_1\cdots m_k\le K$, so $m\le K$ and $k\le (\log K)/(\log 2)$. Heath--Brown proved that $N<4^{4^{s+1}}$ (see \cite{HB}). Hence, we may assume that $s>k$.  Then $M\ge 2^{s-k}$. Now the argument from the proof of Theorem \ref{thm:1} shows that either $q\le K$, or
$$
q^{{\sqrt{M}}-1}\le q^{\alpha}<\sigma(q^{\alpha})<2Mq^{\beta_1+\cdots+\beta_k}<(2KMq)^{(K^2+1)^{k+1}}.
$$
In the second case, we get that $M$ is bounded, hence $s$ is bounded, so $N$ is bounded by Heath-Brown's result.

In the first case, Lemma \ref{lem:Bang} shows that in equations  appearing in the right hand side of equations \eqref{eq:1}, the numbers 
$\lambda_i+1$ are bounded for $i=k+1,\ldots,s$. Let $\Gamma$ be a bound for $\lambda_i$ for $i=k+1,\ldots,s$. For each $\lambda\in \{2,\ldots,\Gamma\}$ and fixed value of $q\le K$, equation
$$
\frac{p^{\lambda+1}-1}{p-1}=p^{\lambda}+p^{\lambda-1}+\cdots+p+1=q^{\beta}
$$
in the unknowns $p$ and $\beta$ has only finitely many effectively computable solutions $(p,\beta)$. Indeed, this follows because if we write $P(t)$ for the largest prime factor of the positive integer $t$, then it is known that if $f(X)\in \Z[X]$ is a polynomial with at least two distinct roots, then $P(f(n))$ tends to infinity with 
$n$ in an effective way. Now we only have to invoke this result for the polynomial $f(X)=(X^{\lambda+1}-1)/(X-1)$ whose $\lambda\ge 2$ roots are all distinct,
and for  the equation $P(f(p))=q\le K$. 
Thus, all primes $p_{k+1},\ldots,p_s$ are bounded, and therefore so is their number $s-k$. Hence, $s$ is bounded, therefore $N$ is bounded by Heath--Brown's result.

\medskip

{\bf Acknowledgements.} The authors thank Luis Gallardo for correcting an error in an earlier version of the paper. F.~L. thanks Professor Arnold Knopfmacher for useful advice.  This paper was written while F.~L. was in sabbatical from the Mathematical Institute UNAM from January 1 to June 30, 2011 and supported by a PASPA fellowship from DGAPA.

\medskip

\noindent MSC2010: 11A25.

\end{document}